\documentclass{amsart}
\usepackage{amsfonts,amssymb,graphicx}
\usepackage{epsfig,amsmath,latexsym}
\usepackage{amscd, amssymb, amsthm, amsmath,eucal, verbatim}
\usepackage{epstopdf}
\newtheorem{theorem}{Theorem}[section]
\newtheorem{lemma}{Lemma}[section]

\newtheorem{corollary}{Corollary}[section]
\newtheorem{claim}{Claim}[section]

\newcommand{\RR}{{\mathbb{R}}}
\newcommand{\ZZ}{{\mathbb{Z}}}


\begin{document}
\title{Minimal and Quasi-Minimal Fibrations of Hyperbolic 3-Manifolds}
\author{Joel Hass}
\address{Department of Mathematics, University of California, Davis}
\email{jhass@ucdavis.edu}
\thanks{Partially supported by NSF grant 1117663 and the Ambrose Monell Foundation}
\begin{abstract}
There are hyperbolic 3-manifolds that  fiber
over the circle but that do not admit  minimal fibrations by minimal surfaces. Furthermore, these manifolds do not
admit fibrations by surfaces that are even approximately minimal.
\end{abstract}
 \maketitle

\section{Introduction} 

\noindent
Thurston showed that a 3-manifold that is a bundle over $S^1$ with fiber a surface of
 genus $g,~ g \ge 2$ and with pseudo-Anosov monodromy admits a
hyperbolic metric. He conjectured that all hyperbolic 3-manifolds
are finitely covered by such bundles \cite{Thurston} and this was later proved by Agol \cite{Agol}. 
One consequence is that understanding the geometry of hyperbolic surface bundles
is central to understanding the geometry of general hyperbolic 3-manifolds.

In this article we describe a construction that gave the first examples 
of fibered hyperbolic 3-manifolds that
do not admit fibrations in which each fiber is a minimal surface. We further
show that the fibers of these  fibrations cannot be made
even approximately minimal, in a sense defined below.

An embedded, orientable, incompressible surface in a closed, orientable,
Riemannian 3-manifold is homotopic to a surface of least area \cite{SacksUhlenbeck, SchoenYau}, and
this surface is either embedded or double covers an embedded one-sided surface \cite{FHS}.
It follows that  in a fibered hyperbolic 3-manifold there is an embedded minimal surface homotopic to a  fiber. 
It is  not known whether in some fibrations the fibers can be isotoped so that all are minimal, giving a minimal fibration.

Minimal foliations of non-compact hyperbolic manifolds do exist,
so the obstruction to a minimal fibration is not local.
Hyperbolic space itself can be minimally foliated 
in many different ways. One such foliation is given by totally
geodesic planes whose limit sets form parallel meridian circles on the sphere at infinity, 
foliating the 2-sphere with its two poles removed by meridians. A large class of
minimal foliations can be constructed by perturbing the meridians
to curves that remain transverse to longitudes and then taking a family of
least area planes that limit to these curves. Such least area planes exist
\cite{ Anderson}. This family forms a foliation, since the least area planes spanning disjoint
curves are disjoint and  an application
of the maximum principle shows that there are no gaps between planes.    
This construction can be made equivariant under a hyperbolic translation that
preserves the two poles, giving a minimal fibration over the circle with planar fibers.
No known construction gives minimal fibers of finite area.
    
Corollary~\ref{nomf} states that 
many hyperbolic 3-manifolds that fiber over $S^1$ do not admit minimal fibrations.
This corollary follows from Theorem~\ref{noqam}, which shows the
non-existence of fibrations whose fibers are even approximately minimal. We now make this
concept precise.

Let $M$ be a smooth Riemannian manifold
and $X \subset M$ a compact surface in $M$, either closed or with boundary. 
We set
$$
 \mathcal{I}_{X}   = \inf \{\mathrm{Area}(G) ~|~G \subset M  \mbox{ is a smooth surface homologous to } X  \mbox{ (rel }  \partial G) \}.
$$
A surface $ F$ is called  {\em area minimizing} if for  any compact subsurface  $X$  of $F$,
$$
\mathrm{Area}(X) = \mathcal{I}_{X}.
$$

Let  $\mu$ and $  \lambda$ be constants  with $ 0 \le \mu,  1 \le   \lambda$.
A surface  $F \subset M$  is  {\em $(\mu, \lambda)$-quasi-area-minimizing} if 
\begin{enumerate}
\item  The mean curvature $H$ of $F$ satisfies $|H|  \le  \mu$,
\item  For any compact subsurface  $X \subset F$, $ \mathrm{Area} (X) \le  \lambda \cdot \mathcal{I}_{X}  $.
\end{enumerate}

\begin{theorem} \label{noqam}
For any constants $\mu < 1$ and $  \lambda \ge 1$ and for any genus $g \ge 2$, there are hyperbolic 3-manifolds that are genus-$g$ surface bundles over $S^1$ and that admit no fibration whose 
fibers are $(\mu, \lambda)$-quasi-area-minimizing surfaces.
\end{theorem}
  
We prove Theorem~\ref{noqam} in  Section~\ref{thmproof}.
 Note that a  $(\mu, \lambda)$-quasi-area-minimizing surface is also a $(\mu', \lambda')$-quasi-area-minimizing surface
 for any $\mu' \ge  \mu$ and $\lambda' \ge  \lambda$.  In particular, we have the following result.
  
\begin{corollary} \label{nomf}
There are hyperbolic 3-manifolds that fiber over $S^1$  that do not admit a minimal fibration.
 \end{corollary}
 \begin{proof} 
The surfaces in a minimal fibration are each area minimizing in their homology class \cite{Sullivan}
and each has mean curvature zero.
So each fiber is a (0,1)-quasi-area-minimizing surface, and it follows that the manifolds
 in Theorem~\ref{noqam} admit no  minimal fibrations.
\end{proof} 

The obstruction to a minimal fibration comes from the  geometry of a hyperbolic manifold
near a short geodesic. Thurston's characterization of hyperbolic surface bundles
points the way to the construction of hyperbolic surface bundles 
that contain arbitrarily short, null-homologous geodesics. Near a short geodesic the
geometry of a hyperbolic  3-manifold
resembles that of a cusp. Direct estimates
show that the area of an incompressible surface going far into a cusp is
larger than that of a homotopic surface that penetrates the cusp less deeply.
One consequence is that such a surface cannot be least area in its homology class. On the
other hand, a leaf of a minimal fibration has no greater area than any homologous surfaces \cite{Sullivan}. 
We construct examples where fibers must go arbitrarily deeply into a neighborhood of a short geodesic, 
but can be homotoped out of the neighborhood. A minimal fibration cannot exist in these manifolds.
In this paper we have not attempted to obtain explicit 
estimates on the lengths of shortest geodesics that provide obstructions. 
Some explicit estimates were given by Huang and Wang \cite{HuangWang2} in the case of a minimal fibration.  
We remark that Wolf and Yu showed the nonexistence of  minimal foliations in the special case where the leaves evolve under  
a geometric flow \cite{Wolf}, a normal flow determined by the principal curvature of the leaves.  
Other results concerning minimal surfaces in cusps can
be found in \cite{CollinHauswirthRosenberg,  Millichap, MoriahRubinstein, Ruberman}.  
The relation between minimal surfaces and short geodesics was studied
by Breslin \cite{Breslin}. The extension of the investigation of this question from minimal to quasi-minimal surfaces
was not previously studied.

The existence of minimal fibrations is closely related to the question
of whether there exist non-isolated minimal surfaces in hyperbolic 3-manifolds,
and to the existence of unstable minimal surfaces.
A foliation of a Riemannian 3-manifold with 2-dimensional leaves is {\em taut}
if each leaf intersects a closed transversal curve. Fibrations give one class of
examples of taut foliations. Taut foliations were studied by Novikov, who
showed among other results that each leaf of such a foliation is incompressible
\cite{Novikov}. Sullivan showed that a smooth foliation of a  3-manifold is taut if and only if there is a
Riemannian metric on the manifold in which each leaf is a minimal surface \cite{Sullivan}. A consequence
of Sullivan's theorem is that a surface bundle over $S^1$ admits some
Riemannian metric in which each fiber is minimal. Corollary ~\ref{nomf} shows
that Sullivan's construction is often not compatible with a hyperbolic
metric.

\noindent
{\bf Acknowledgements}
The results in this paper originated in conversations between the author and Bill Thurston
at the 1984 Durham Symposium on {\em Kleinian groups, 3-Manifolds, and Hyperbolic
Geometry}.  The generalization to quasi-area-minimizing surfaces was written while the author was visiting the
School of Mathematics, Institute for Advanced Study in 2015. 

\section{Surfaces in cusps}

We review some standard facts about the geometry of cusps of hyperbolic 3-manifolds.
An  {\em ideal hyperbolic cusp} $C$ is a hyperbolic 3-manifold homeomorphic to $T^2 \times \RR$,
obtained as the quotient of $\mathbb H^3$ by a parabolic subgroup $\Gamma$ of $\mathrm{PSL}(2,\mathbb C)$ isomorphic to $\ZZ \oplus \ZZ$.
In the upper half-space model of $\mathbb H^3$ 
the generators of $\Gamma$ act as translations of the $xy$-plane. 
A fundamental domain for the cusp is $Q \times (0,\infty)$,
where $Q$ is a parallelogram on the $xy$-plane. The cusp is foliated by flat horotori $T_s$, 
with $T_s$ covered by the horosphere $ \{ z= s \} $.
 
Conjugating $\Gamma$ by  the isometry $z \to \lambda z, ~\lambda \in \RR^+$, takes the horosphere
$ \{ z= s \} $ to the horosphere $ \{ z= \lambda s \} $. 
We can use this conjugacy to arrange that the horotorus whose shortest nontrivial curve has length one lifts to the plane $ \{ z=1 \} $ in 
the upper half-space model for $\mathbb H^3$, 
so that $T_1$ has injectivity radius 1/2. 
The lengths of curves on the horotori $T_s$ 
decrease linearly with $s$, so that the shortest curve on  $T_s$ has length $1/s$, or equivalently,
the injectivity radius of $T_s$ satisfies $ i_s   = 1/(2s) $ for all $s\ge 1$.
Let $C_{[a,b]}$ denote the portion of the cusp consisting of $T_s$ with $a\le s \le b$  
and $C_{[a,\infty)}$ the end of the cusp cut off by $T_a$.

\begin{lemma} \label{trivial1}
Let $D$ be a smooth disk in an ideal cusp $C$ with $\partial D \subset T_s$ whose boundary has length  $ l$.
Then either  $D  \subset C_{[s,s +(sl^2)/4]}$ or $D$ has an interior point where its
mean curvature satisfies $|H| \ge 1$.
\end{lemma}
\begin{proof}
We work in the upper half-space model.
The disk $D$ lifts to a disk $\tilde D$ in $\mathbb H^3$
and its boundary curve lifts to a curve $\gamma$ in the horosphere $ \{ z= s \} $.
Let $E$ be a  disk in the horotorus  $ \{ z= s \} $ of  radius $l/2$ (in the induced flat horotorus metric),  centered on a point of 
$\gamma$.
Since $\gamma$ lies on the horotorus and has hyperbolic length $l$, its length in the flat horotorus
metric is also $l$, and it lies in the interior of $E$. 
In the Euclidean metric on the upper-half-space, $E$ has radius $ls/2$ and is
the intersection of a horoball $B$ of $\mathbb H^3$ with the horosphere $ \{ z= s \} $.
A calculation shows that the horoball  $B$ has Euclidean height $h = s + (sl^2)/4$,
as indicated in Figure~\ref{null}.
 
 \begin{figure}[htbp] 
    \centering
    \includegraphics[width=3in]{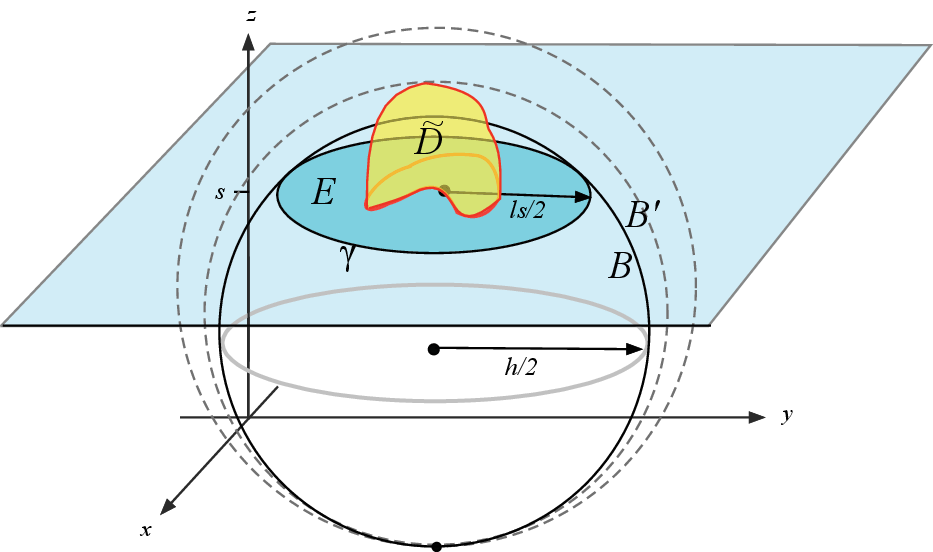} 
    \caption{A minimal disk bounded by a curve  of hyperbolic length $l$ 
    that lies on the horosphere $ \{ z= s \} $ is contained inside a horoball of height $ h=s + ( sl^2)/4$.
    Indicated distances are Euclidean.}
    \label{null}
 \end{figure}

If the interior of $\tilde D$ meets $ \{ z= s \} $, then it tangentially meets some horosphere $ \{ z= s' \} $  with $s' \le  s$ and $s'$ minimal. 
Then it meets $ \{ z= s' \} $  tangentially, without crossing it.  Since the mean curvature of a horosphere $ \{ z= s' \} $ is one,
the mean curvature of $\tilde D$ at the  point of tangency satisfies $|H| \ge 1 $. It follows that if the mean curvature of
 $\tilde D$ is less  than 1 then $D \subset  C_{[s,\infty)}$.
 
If $\tilde D$ is not contained in $B$  then it is contained in a largest horoball $B'$ such that $B \subset B'$ and $\tilde D$ meets $\partial B'$ without crossing it. Again $\tilde D$
 has mean curvature $|H| \ge 1$  at the point where it meets $\partial B'$. Thus  if the mean curvature of $\tilde D$ is less than one, then it lies in the slab $ \{ s \le z \le  s + (sl^2)/4  \} $ and $D$ lies in its 
 quotient $ C_{[s,s +(sl^2)/4]}$.
  \end{proof}

\begin{lemma} \label{trivial2}
Let $D$ be a smooth disk in an ideal cusp $C$ with $\partial D \subset T_s$ and length$(\partial D) = l $ with  $l <2i_s$ and $s \ge 1/2$.
Then either  $D  \subset C_{[s,2s)}$ or $D$ has an interior point where its
mean curvature satisfies $|H| \ge 1$.
\end{lemma}
\begin{proof}
Assume $|H|< 1$ and apply Lemma~\ref{trivial1}. Since $i_s = 1/(2s)$ and $ l <2i_s$, 
$$
s +  \frac{sl^2}{4}< s +s (i_s)^2 = s+ \frac{1}{4s} \le s + 1/2  < 2s.
$$
\end{proof}

We now consider a  complete, finite-volume hyperbolic 3-manifold $M$  with a  cusp $C$.
The fundamental group of the cusp is isomorphic to a $\ZZ \oplus \ZZ$-subgroup of 
$\pi_1(M)$  generated by parabolic elements. With the  upper-half space model representing the universal cover of $M$, 
this  $\ZZ \oplus \ZZ$-subgroup can be
conjugated to fix infinity, and thus act as translations.
The covering of $M$ corresponding  to this subgroup is an ideal cusp $C$, which
is foliated by horotori whose injectivity radius approaches 0 as they approach the end of the cusp. 
As before we parametrize the  horotori of
$C_{[s, \infty)}$ so that $T_s$ has injectivity radius $1/(2s)$.
The projection of $C$ to $M$ is injective on $C_{(m, \infty)}$ for some 
$m \le 1$  \cite{Adams}. 
Thus the submanifold $C_{(m, \infty)}$ of $M$ is isometric to a  
submanifold of an ideal cusp cut off by a horotorus.
The  {\em maximal horotorus} $T_m$ lies in the boundary of the cusp. 
It is self-tangent at some number of points, but is the limit of embedded horotori in the cusp. 
For any $s$ with $s>m$ the cusp $C_{(s, \infty)} \subset M$ is isometric to an end of an ideal
cusp and so the results and terminology of Lemmas~\ref{trivial1} and \ref{trivial2} apply in $M$.

We  now  bound how far a surface  can go into a cusp when it is both $(\mu,\lambda)$--quasi-area-minimizing and incompressible.
    
\begin{lemma} \label{surfaceincusp}
Let  $M$ be a finite volume hyperbolic manifold with cusp $C = C_{[1,\infty]}$,   $\mu $ and $\lambda$ constants with  $0  \le \mu < 1, 1 \le  \lambda  $,
and  $F $ a smooth, compact, properly embedded, incompressible,
$(\mu, \lambda)$-quasi-area-minimizing surface in $M$.
If   $s \ge  2 \lambda \mathrm{Area}(T_1) $ then $F \cap C_{[s, \infty)} \subset  C_{[s,   4s]}$. 
 \end{lemma}
\begin{proof} 

We first consider  the case where the intersection $F \cap T_s$ is transverse, and let $F_s$ denote $F \cap C_{[s, \infty)} $.
We can assume $F_s$ is connected, as if not we can consider one component at a time.
Moreover $F_s $  has non-empty boundary, as if closed it would meet a horotorus  $T_u  \subset C_{[s, \infty)} $ with $u$ minimal. They would meet at an interior tangency point of $F_s$  at which its mean curvature satisfies $H(F_s)  \ge 1> \mu$.

Since $F_s$ can be homotoped (rel boundary)  into $T_s$, it follows that
 the  curves in $F  \cap  T_s$ are null-homologous and separate $T_s$ into two subsurfaces.
Each of these has area  less than  
$$  
 \mathrm{Area}(T_s)  =  \frac{\mathrm{Area}(T_1)}{s^2 }. 
 $$
One of these subsurfaces  is  homologous (rel boundary) to  $F_s$.
Comparing areas and using the assumption that $2 \lambda \mathrm{Area}(T_1)  \le s $,  
we see that
$$
 \mathcal{I}_{F_s}  <    \mathrm{Area}(T_s)   =\frac{\mathrm{Area}(T_1)}{ s^2}    \le  \frac{1}{ 2 \lambda s}.
 $$ 
 
Sard's theorem implies that for almost all  $z>s$ the intersection  $F \cap T_z$ is 
transverse, and forms a collection of smooth curves.
Suppose  that for almost all  $z \in [s , 2s)$, the length of the intersection $F  \cap T_z$
satisfies $\mathrm{length}(F \cap T_z)  \ge 2i_z = 1/z$.
Then the area of $F_s$ in the cusp region 
$C_{[s, 2s)} $ can  be bounded from below using the co-area formula,
\begin{equation*}
\mathrm{Area}(F_s  )  \ge   \mathrm{Area}(F_s \cap C_{[s, 2s)})   \ge
\int_{s}^{ 2s} \mathrm{length}(F \cap T_z)   \frac{1}{z} ~ dz \ge  \int_{s}^{2s} \frac{1}{z^2}  ~ dz  = \frac{1}{2s}.
\end{equation*}
Since $F_s$ is $(\mu, \lambda)$-quasi-area-minimizing,
$$
 \mathrm{Area}(F_s)   \le   \lambda  \mathcal{I}_{F_s} <    \lambda    \frac{1}{ 2 \lambda s} = \frac{1}{2s}.
 $$
which is a contradiction.
So    $\mathrm{length}(F_s \cap T_z)  <  2i_z = 1/z$  for some $z \in [s, 2s) $.
For this $z$,  $F_s \cap T_z$ consists of a collection of null-homotopic curves, each shorter than
$2i_z$.   Since  $F$ is incompressible, each of these bounds a subdisk of $F$.
Lemma~\ref{trivial2} then implies that each such subdisk is contained in $C_{[z,2z)}$ and therefore that
$F_s \subset C_{[s,2z)} $.  Since $2z < 4s$ we have $F_s  \subset   C_{[s,4s)}$ and
$F_s \cap T_{4s} = \emptyset$. 
 
Now suppose that $F \cap T_{s}$  is not transverse.
For any $\epsilon >0$ there is an $s_1 \in [s, s + \epsilon]$
such that  $F \cap T_{s_1}$ is transverse.  By the previous argument
$ F \cap  T_{4s_1} = \emptyset$, so $F$ is disjoint from
$T_{ 4(s + \epsilon)}$ for any $\epsilon >0$.  It follows that $F \cap C_{[s, \infty)} \subset  C_{[s,   4s]}$.  
 \end{proof}
 
The next two lemmas estimate for how far a minimal surface can reach into a cusp.
Unlike the previous lemmas, they do not assume that the surface is incompressible,
but instead use the assumption that it has  mean curvature zero. The  monotonicity formula for minimal
surfaces shows
that the area  of $F \cap C_{[s, 2s]}$ is bounded below by $3/32s$. 
This is used in Section~\ref{thmproof}  to
 establish the quasi-area-minimizing property.
 
\begin{lemma} \label{monotonicity}
Let  $M$ be a  hyperbolic 3-manifold with cusp $C$,
$s \ge  8 \mathrm{Area}(T_1) $, and
$F_s \subset C_{[s, 2s]}$ a smooth, properly embedded
minimal surface.
Then either $F_s \cap T_u = \emptyset$ for some $u \in  (s, 2s)$ or
$\mathrm{Area}(F_s) >  3/(32s) $.
 \end{lemma}
 
\begin{proof} 
By standard monotonicity estimates
the area of a minimal surface   passing
through the center of a ball of radius $r$ in hyperbolic space is at least as large as a 
hyperbolic disk of radius $r$, namely $4 \pi \sinh^2( r/2) $  \cite{HS}. We use this to make
some rough approximations to the area of $F$, assuming that 
$F_s \cap T_u \ne \emptyset$ for all $u \in  (s, 2s)$.

A ball of Euclidean diameter one in the upper-half space model that lies
between $z = s$ and $z=s+1$  can be embedded  in the cusp between 
$T_s$ and   $T_{s+1}$ so  its Euclidean center lies at any point in
$T_{s+1/2}$.
Such a ball has hyperbolic diameter given by
$$
\int_s^{s+1} 1/z ~dz = \ln (s+1) - \ln (s) = \ln (1+ 1/s) \ge \frac{1}{s}  - \frac{1}{2s^2}.
$$
Since $ s  \ge  8 \mathrm{Area}(T_1) \ge  4 \sqrt{3} > 4$ we have that the 
hyperbolic diameter is greater than $1/(2s)$.
Thus the radius of this ball is greater than $1/(4s)$ and 
the monotonicity estimate tells us that a minimal surface passing through the ball's center
has area inside the ball that is greater  than 
$$
4 \pi \sinh^2(\frac{1}{4s}) >  \frac{\pi }{16s^2} > \frac{3}{16s^2}.
$$

We can apply this estimate to
$s$ disjoint balls in the cusp of Euclidean radius one, with one ball lying between 
each adjacent pair of the planes $z = s, z=s+1, \dots z=2s$,
and with each ball centered on a point of $F$.
If $F_s$ is not disjoint from any intermediate $T_u$, then this results in a lower bound on the area of $F_s$ of 
$$
 \mathrm{Area}(F_s) >\frac{3}{16}(\frac{1}{s^2}  + \frac{1}{ (s+1)^2}  \dots  + \frac{1}{ (2s)^2} ) >\frac{3}{16}  (\frac{1}{s}  -\frac{1}{2s+1}  )  >  \frac{3}{16}  (\frac{1}{s}  -\frac{1}{2s}  )  =  \frac{3}{32s}
$$
as claimed.
 \end{proof}

We  apply Lemma~\ref{monotonicity} to bound how far an area-minimizing surfaces extends into a cusp.
 
\begin{lemma} \label{minimizer}  
Let  $M$ be a  hyperbolic 3-manifold with cusp $C = C_{[1, \infty)}$,  
$s$ a constant satisfying $s > 14 \mathrm{Area}(T_1) $ and
$F$ a smooth, compact, embedded, area-minimizing surface in  $M$.
Then $F  \cap T_w = \emptyset$ for  $w > 2s$.
 \end{lemma}
\begin{proof} 
 If  $F \cap T_{s} $ or  $F \cap T_{2s} $  are not transverse, 
 replace $s$ with a slightly larger $s'$ with $s < s' < w/2 $, for which 
both $F \cap T_{s'} $ and  $F \cap T_{2s'} $ are transverse.
Now  $F$  is compact and homotopic out of the cusp, and therefore  $F \cap C_{[s, \infty]} $ is separating 
in $ C_{[s, \infty]} $ and
$F_s = F \cap C_{[s, 2s]} $ is properly embedded and separating in $C_{[s,2s]} $.
If  for all $u \in  [s, 2s]$ $F_s \cap T_{u} \ne \emptyset$  then Lemma~\ref{monotonicity} implies that
$$
 \frac{3}{32s}   <  \mathrm{Area}(F_s).
$$
Since $F_s$ is separating in the cusp,  the curves of intersection of $F \cap T_s $ separate $T_s$ into two subsurfaces
and similarly the curves in $F \cap T_{2s} $ separate $T_{2s}$.
We can compare the area of  $F_s$ to that of the homologous surface  with the same boundary formed by the union of subsurfaces of $T_{s}  \cup T_{2s} $. 
Then
$$
\mathrm{Area}(F_s) < \mathrm{Area}(T_s) + \mathrm{Area}(T_{2s})
 =  \frac{5\mathrm{Area}(T_1)}{4s^2}.
$$
Combining these inequalities  gives 
$$
 \frac{3}{32s} < \mathrm{Area}(F_s)  <\frac{  5  \mathrm{Area}(T_1)}{4s^2}
$$
implying that
$   s  <(40/3) \mathrm{Area}(T_1) < 14  \mathrm{Area}(T_1)$. This contradicts the assumption that $s > 14 \mathrm{Area}(T_1) $, so 
 $F \cap T_u = \emptyset$ for some $u \in  [s, 2s]$.  We also have that $F \cap T_u = \emptyset$ for all $w>u$,
 since otherwise   $F$ would meet a horotorus at a point where its mean curvature would be greater than 0. In particular
$F  \cap T_w = \emptyset$ for  $w > 2s$.
 \end{proof} 
 
\section{Bundles with short geodesics} \label{sec:bundles}

In this section we use Thurston's Dehn Surgery Theorem \cite{Thurston, BenedettiPetronio, HodgsonKerckhoff, MoriahRubinstein} 
to construct a  sequence of hyperbolic 3-manifolds $ M_j$  that fiber 
over the circle and  limit to a hyperbolic 3-manifold with a cusp.
These manifolds contain embedded geodesics that are homotopic into a fiber and whose length approaches zero
as $j \to \infty$.

The complement of a simple closed geodesic in a closed hyperbolic 3-manifold $ M_0$
is the interior of  a manifold $M$ with torus boundary.
The manifold $M$ is atoroidal and irreducible and therefore satisfies the hypothesis of
Thurston's geometrization theorem for Haken manifolds. 
The absence of essential spheres and tori can be seen by
considering the lift of such a surface to the complement of a collection
of geodesic lines in the universal cover of $M$.  Alternately spheres and non-peripheral tori can
be ruled out by noting that it is possible to explicitly construct a negatively curved
metric on the complement \cite{Kojima}. It follows that
$M$ has a complete, finite volume hyperbolic metric
with a single cusp.

The notion of geometric convergence allows for comparing
a non-compact manifold to a sequence of compact manifolds that resemble it on increasingly large
subsets.  A sequence of pointed Riemannian manifolds $(M_j, p_j)$ converges geometrically to a pointed manifold
$(M, p)$  if for every compact subset $K \subset M$ with $p \in K$ there is a sequence of smooth maps 
$f_j: (K , p)  \to   (M_j  , p_j) $  such that  the pulled back Riemannian metrics $f_j^*(g_j)$ 
converge smoothly to the hyperbolic metric on $K \subset M$  \cite{BenedettiPetronio, BiringerSouto}.
 This is roughly illustrated in Figure~\ref{fig:cusp1} for $K$ the compact submanifold of $M$ cut off by a horotorus $T_s$.
 
 \begin{figure}[htbp] 
    \centering
    \includegraphics[width=4in]{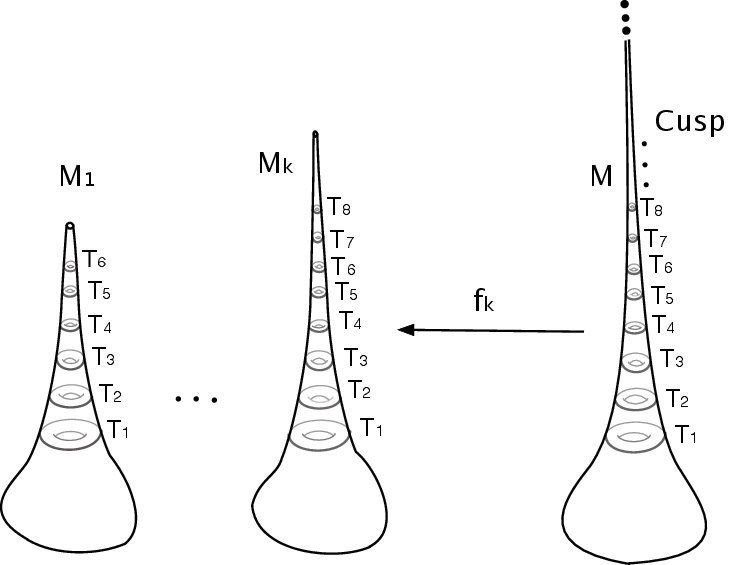} 
    \caption{Geometric convergence implies that  $K_s$, the compact submanifold of $M$ below a fixed horotorus $T_s$, increasingly resembles the submanifold of $M_j$  below $f_j(T_s)$} 
    \label{fig:cusp1}
 \end{figure}

We now  construct a  
sequence of manifolds $M_j$, each of which fibers over the circle, that geometrically converge to a cusped manifold $M$.
We fix a genus $g>1$.
\begin{lemma} \label{construction}
There exists a sequence of hyperbolic 3-manifolds $M_j,  j = 1,2,\dots$ with the following properties:
\begin{enumerate}
\item  For each $j \ge 1$, $M_j$ fibers over the circle with fiber a surface of genus $g$.
\item The manifolds $M_j$  converge geometrically
to a cusped hyperbolic manifold $M$ as  $j\to \infty$.
\item  There is a sequence of
closed geodesics $\gamma_j \subset M_j$,
homotopic into a fiber of $M_j$, with 
$ \lim_{j \to \infty} \mathrm{length}(\gamma_j) = 0.$
\item  The horotorus $T_1$ with injectivity radius $1/2$ that bounds the cusp of $M$  is mapped by
$f_j$ to a torus $f_j(T_1)$  that encloses a radius $R_j$  tubular neighborhood of
the geodesic $\gamma_j$, with $\lim_{j \to \infty}  R_j  = \infty$.
\end{enumerate}
\end{lemma}
\begin{proof}
Thurston showed that a  surface bundle over $S^1$ with 
fiber a surface of genus $g \ge 2$  is hyperbolic if the monodromy is pseudo-Anosov.
A construction of Penner shows that if $C$ and $D$ are two  collections of disjoint embedded
 essential closed curves (with no parallel components) in an oriented surface
$F$ that intersect efficiently and fill $F$, and if $\phi$ is a product of positive Dehn twists along
$C$ and negative Dehn twists along $D$ that twists along each curve at least once, then $\phi$ is pseudo-Anosov \cite{Penner}.
Take a closed surface $F_g$ of genus $g \ge 2$ and on $F$  a pair of curve collections as above
 in which $c_1 \in C$. 
Consider the sequence of pseudo-Anosov maps  $\phi_j$
obtained by composing a fixed number of positive Dehn twists 
along each curve in $C$ other than $c_1$, a fixed number of
negative Dehn twists along each curve in $D$, and finally $j$ positive twists along $c_1$. 
 Then  $\phi_j$  is pseudo-Anosov for each $j \ge 1$, and 
 the manifold  $M_j$ formed by constructing a surface
 bundle over $S^1$  with monodromy $\phi_j$ is hyperbolic.
  
The bundle $M_j$  is obtained from $M_1$ by  
$1/(j-1)$-Dehn surgery on $c_1$. 
This surgery first removes a solid-torus neighborhood $N$ of $c_1$ in $M_1$, giving a 3-manifold $M$ that admits a 
complete hyperbolic metric with a cusp. We call a curve on $\partial N$ 
lying in a fiber a {\em longitude} and a curve bounding a disk in $N$  a {\em meridian}. The surgery
attaches a solid-torus $S^1 \times D^2$ to $\partial N$ so that a meridian
is mapped to a curve homotopic to one meridian and $(j-1)$-longitudes. 
Thurston's Dehn Surgery Theorem shows that  the hyperbolic manifolds $M_j$ 
converge geometrically to $M$ as $j \to \infty$,
 that the core curve of the solid torus attached by the Dehn surgery
is isotopic to a closed  geodesic $\gamma_j \subset M_j$ and that $ \lim_{j \to \infty} \mathrm{length}(\gamma_j) = 0$.
Moreover for any fixed $s$, the maps $f_j$ carry a horotorus  $T_s$ in the cusp of $M$ to the boundaries of radius $R_j$  tubular neighborhoods   of
the geodesic $\gamma_j$, with $\lim_{j \to \infty}  R_j  = \infty$.   
\end{proof}  
 
 \section{Non-existence of quasi-area-minimizing fibrations} \label{thmproof}

We now present the proof of  Theorem~\ref{noqam}.
Given constants $0  \le \mu < 1, 1 \le  \lambda  $,
we show that when $j$ is sufficiently large it is not possible to isotop all
the fibers on the hyperbolic manifolds $M_j$ constructed in Lemma~\ref{construction}
to be simultaneously $(\mu, \lambda)$-quasi-area-minimizing.

We let $M(s)$ denote the compact manifold that is the complement in $M$ of $C_{(s,\infty)}$, for each $s \ge 1$,
and  $M_j(s)$ denote  $f_j(M_s)$. 

\begin{proof} [Proof of Theorem \ref{noqam}]
Assume that  
the  manifolds $M_j $  in Lemma~\ref{construction}
admit a $(\mu, \lambda)$-quasi-area-minimizing fibration for infinitely many values of $j$. We will derive a contradiction.

In the construction of  $M_j$ in Lemma~\ref{construction} 
the torus $ f_j(T_1)$ separates $M_j$ into a solid torus
$N_j \subset M_j$   with core a geodesic $\gamma_j $ and 
 a compact manifold $M_j(s)$.
Moreover  the radius of $N_j$ satisfies  $R_j  = d(f_j(T_1),  \gamma_j) \to \infty$. 

Fix the constant $s_0 = 16 \lambda  \mathrm{Area}(T_1)$.
The sequence of manifolds $\{ M_j \}$ geometrically converges to the cusped hyperbolic manifold $M$,
 so for $j$ sufficiently large the maps
  $f_j :  M(2s_0)  \to  M_j (2s_0) $ are embeddings that converge to an isometry.
So the pulled back hyperbolic metrics $f_j^*(g_j)$ 
converge smoothly to the hyperbolic metric on $M(2s_0)$.  

We now look for  a fiber in $M_j$  that comes close to $\gamma_j \subset M_j$ but does not intersect $\gamma_j $.
It is not clear that such a fiber exists, since it is possible that every fiber might intersect $\gamma_j $.
However we can always find a component $X_j$ of the intersection of a fiber with the solid torus $N_j$
 that has this property.
 
\begin{claim}
For  $j$ sufficiently large, there is a fiber $F_j \subset M_j$
and a component $X_j$ of  
$F_j  \cap N_j$ 
such that $X_j \cap f_j (T_1) \ne \emptyset$,  $X_j \cap f_j (T_{s_0+2}) \ne \emptyset,$  but 
$X_j \cap f_j (T_{s_0+3}) = \emptyset$. Moreover $X_j$ is either
a boundary parallel disk or an essential annulus in $N_j$.
\end{claim}

\begin{proof}
The surface bundle $M_j$ has an infinite cover $\tilde M_j $ 
homeomorphic to the product  $F_g \times  \RR$ of a 
surface of genus $g$ with $\RR$.
Each fiber in  $M_j $ lifts to a fiber in  $\tilde M_j $.  
The curve $\gamma_j $ is homotopic into a fiber, and so lifts to a 
loop $\tilde \gamma_j \subset \tilde M_j $
 homotopic into a fiber of $ \tilde M_j $.
Therefore there are fibers in $\tilde M_j $ that intersect $\tilde \gamma_j$ 
and fibers that are arbitrarily far from  $\tilde \gamma_j$. 
The solid torus $N_j $ lifts to a solid torus $ \tilde N_j \subset \tilde M_j $ of radius $R_j$.
By continuity $ \tilde N_j$ intersects fibers in $ \tilde M_j $ in components
whose distance from $\tilde \gamma_j$ varies from zero to $R_j$.
Let $\tilde X_j$ be a component of $\tilde F_j \cap \tilde N_j$
that intersects a lift of  $ f_j (T_{s_0+2}) \subset N_j$ to $ \tilde M_j $  but not a lift of 
$ f_j (T_{s_0+3})$ and let $X_j$ be the projection of $\tilde X_j$ to $M_j$.
Then $X_j $ intersects  $ f_j (T_{s_0+2})$  but not
$ f_j (T_{s_0+3})$, as claimed.  Note that  $X_j $ must intersect   $f_j (T_1)$ as
otherwise it would be contained in the interior of   $ f_j(C_{[1, s_0+3)})$
and tangent to   $ f_j (T_{s}) $ for some $s>1$.  This would imply that
the mean curvature $|H(X_j)| $ of  $X_j $ at the tangency point is at least as large as
 $ |H( f_j (T_{s}) )|$, which is larger than $\mu$,
while $X_j $ has mean curvature  $|H(X_j)| \le \mu $.
Thus $X_j \cap f_j (T_1) \ne \emptyset$,  $X_j \cap f_j (T_{s_0+2}) \ne \emptyset$  and 
$X_j \cap f_j (T_{s_0+3}) = \emptyset$. 

It remains to show that $X_j$ is  either a boundary parallel disk or an essential annulus in 
$f_j(C_{[1,\infty)})$.

Suppose first  that $\alpha \subset \partial   X_j$ is a null homotopic curve on
$\partial  N_j$. Since the fiber  $ F_j $ is incompressible and $X_j \subset F_j$,  $\alpha$ is the boundary
of a disk $D \subset F_j$.  We claim that $D \subset N_j$. If not, then
$D$ protrudes outside $ f_j(C_{[1,\infty)})$ and intersects the interior of $M_j(1)$.  
Now consider the lift $\tilde D$ of $D$ to
the cover of $M_j$ given by the cyclic  subgroup of $\pi_1(M_j)$ generated by
$\gamma_j$. The disk $\tilde D$ has boundary on 
the boundary of an $R_j$ neighborhood of $\tilde \gamma_j$ 
and has an interior point in the  complement of this neighborhood.  At a point where
it is farthest away from $\tilde \gamma_j$ the mean curvature of $\tilde D$ is
greater than the mean curvature of the boundary of a constant radius
tubular neighborhood of the geodesic $\tilde \gamma_j$. The mean curvature of
such a tubular neighborhood boundary is always greater than one.
Since $D \subset  F_j$ has mean curvature $|H| \le \mu < 1$, this gives a contradiction unless $D  \subset N_j $. 

We now show that $ X_j$ is incompressible in $N_j $. 
If not, there is a nontrivial
compressing disk $G$ for $ X_j$ with $G \subset N_j$. Since $ F_j$ is incompressible, $\partial G$
bounds a disk on $ F_j$, which must protrude out of $N_j$.  But  such a disk must have a point where
its mean curvature is greater than one, as  in the preceding argument, contradicting that $F_j$ has mean curvature
$|H(F_j) | \le \mu < 1$.
So $X_j$ is incompressible in $N_j $ and disjoint from $  f_j (T_{s_0+3}) $  and therefore from $\gamma_j$. 
An incompressible surface in a solid torus  that misses the core of the solid torus
 is a disk or annulus and boundary parallel.
So $X_j$ is either  a boundary parallel disk or a boundary parallel essential annulus in $N_j$,
and $X_j$ satisfies the  claimed properties.
\end{proof}  
  
The geometric convergence of $M_j$ to $M$ implies that the area forms and the second fundamental forms of the 
surfaces $X_j$ at $p_j \in X_j$ converge to those of $Y_j = f_j^{-1}(X_j)$ at  $ f_j^{-1}(p_j) \in Y_j$.  
Given constants  $(\mu', \lambda')$ satisfying  $\mu < \mu' < 1$ and  $\lambda < \lambda' $,  
we  show that $Y_j$ is a  $(\mu', \lambda')$-quasi-area-minimizing surface for  $j$ sufficiently large.

Since the mean curvature of $X_j$ satisfies  $|H(X_j) | \le  \mu < 1$ and $ H(X_j) \to H(Y_j)$
we have that  $|H(Y_j) | \le  \mu' < 1$ for $j$ sufficiently large. 

We next check that $Y_j$  is $\lambda'$-quasi-area-minimizing for $j$ sufficiently large.
Take any compact subsurface  $U_j \subset Y_j$ and a 
surface $V_j \subset M$ homologous to  $U_j \mbox{ (rel }  \partial U_j)$.
We need to show $  \mbox{Area} (U_j)  \le  \lambda'  \mbox{Area} (V_j)  $.
It suffices to choose $V_j$ to be an area-minimizing surface with boundary equal to $\partial U_j$.
Then Lemma~\ref{minimizer}
implies that  $V_j \subset  C_{[1,2s_0]}$.
If  $  \mbox{Area} (U_j)  > \lambda' \mbox{Area} (V_j)  $, then since the $f_j$ converge to an isometry
on $C_{[1,2s_0]}$,
 $\mbox{Area} (f_j(U_j)) > \lambda \mbox{Area} (f_j(V_j))$ for $j$ large.  This contradicts the assumption that 
 $X_j$ is  $(\mu, \lambda)$-quasi-area-minimizing.
We conclude that  for $j$ sufficiently large $Y_j$ is $(\mu', \lambda')$-quasi-area-minimizing.

The surface  $Y_j \subset C_{[1,\infty]}$ is either a boundary parallel disk or a boundary parallel  annulus, and therefore incompressible in $C_{[1,\infty]}$.
Lemma~\ref{surfaceincusp} implies that if   $s \ge  2 \lambda \mathrm{Area}(T_1) $ then $Y_j \cap C_{[s, \infty)} \subset  C_{[s,   4s]}$. 
Taking  $s =  4 \lambda \mathrm{Area}(T_1) = s_0/4$, Lemma~\ref{surfaceincusp}  implies that  $Y_j \cap C_{[s, \infty)} \subset  C_{[s,   4s]} = C_{[s,   s_0]}$.
In particular  $Y_j$  does not intersect $ T_{s_0+1} $, for large $j$. 
But the maps $f_j$ converge smoothly to an isometry on $M(2s_0)$, and
since  $X_j \cap f_j (T_{s_0 +2}) \ne \emptyset $, for large $j$ we have
$Y_j \cap  T_{s_0+1} \ne \emptyset$, so that
$Y_j$ intersects $ T_{s_0+1} $ for all $j$ sufficiently large.

This gives a contradiction.  We conclude that at most finitely many of the sequence of 
manifolds $M_j$ admit a fibration by $(\mu,  \lambda)$-quasi-area-minimizing fibers, proving Theorem~\ref{noqam}.
\end{proof}

 \noindent
{\bf Remark.} For mean curvature zero fibers, which are area minimizing,  curvature estimates show that
a subsequence of the  surfaces $Y_j = f_j^{-1}(X_j)$ converges to a limiting surface $Y$~\cite{Schoen}.
This allows for a simpler proof of Theorem~\ref{noqam}.
However such curvature bounds do not extend to the case of  quasi-area-minimizing surfaces.
 
\section{Questions}

Some basic questions remain open.
 \begin{enumerate} 
\item Is there a closed hyperbolic 3-manifold with a  fibration whose fibers have mean curvature $|H| \le 1$?  Note that no fibration exists with the
absolute value of the principal curvatures bounded above by one,  since the limit curve of a fiber is a space-filling curve.
 \item Is there a closed hyperbolic 3-manifold with a minimal fibration?
  \item Is there a closed negatively curved 3-manifold with a minimal fibration?  Note that the arguments of this paper are stable under perturbation of
  the hyperbolic metric, so that there exist fibered 3-manifolds that admit no metric with
  pinched negative sectional curvature close to -1  and a  quasi-area-minimizing fibration.
 \end{enumerate}

\end{document}